\documentclass{amsart}
\usepackage{amssymb}
\usepackage{amscd}
\usepackage{amsfonts}
\usepackage{latexsym}
\usepackage[all]{xy}
\usepackage{verbatim}

\newtheorem{theorem}{Theorem}[section]
\newtheorem{lemma}[theorem]{Lemma}

\newtheorem{corollary}[theorem]{Corollary} 
\theoremstyle{definition}  
\newtheorem{definition}[theorem]{Definition}
\newtheorem{example}[theorem]{Example}

\newtheorem{remark}[theorem]{Remark}

\newcommand{\Tr}{\text{Tr}}
\newcommand{\id}{\text{id}}

\newcommand{\FPdim}{\text{FPdim}} 
 
\newcommand{\Irr}{\text{Irr}}

\newcommand{\Hom}{\text{Hom}}

\newcommand{\C}{\mathcal{C}}

\newcommand{\Z}{\mathcal{Z}}

\renewcommand{\O}{\mathcal{O}}

\newcommand{\be}{\mathbf{1}}

\newcommand{\imply}{\Rightarrow}

\renewcommand{\be}{\mathbf{1}}

\newcommand{\BZ}{{\mathbb Z}}
\newcommand{\BQ}{{\mathbb Q}}
\newcommand{\BR}{{\mathbb R}}
\newcommand{\fA}{{\mathbb A}}

\newcommand{\CC}{{\mathbb{C}}}

\newcommand{\otz}{\otimes_{\BZ}}

\begin{document}

\title{On formal codegrees of fusion categories}

\author{Victor Ostrik}
\address{V.O.: Department of Mathematics,
University of Oregon, Eugene, OR 97403, USA}
\email{vostrik@math.uoregon.edu}

\begin{abstract}
We prove a general result which implies that the global and Frobenius-Perron dimensions
of a fusion category generate Galois invariant ideals in the ring of algebraic integers.
\end{abstract}

\date{\today}
\maketitle  

\section{Introduction}
The goal of this note is to give some new restrictions on the Grothendieck rings of (multi-)fusion 
categories. Let $k$ be an algebraically closed field of characteristic zero and let
$\fA \subset k$ be the subring of algebraic integers.
Recall (see \cite{ENO}) that a multi-fusion category $\C$ over $k$ is a $k-$linear semisimple
rigid tensor category with finite dimensional $\Hom$ spaces and finitely many simple objects;
such category is said to be fusion category if its 
unit object $\be$ is simple. 
The Grothendieck ring $K(\C)$ of a multi-fusion category $\C$
is an example of a {\em based ring} in a sense of \cite{Lu1}, see \S \ref{basedS} below. 
In particular for any irreducible
representation $E$ of $K(\C)\otz k$ its {\em formal codegree} $f_E\in k$ is defined, see \S \ref{basedS}.
It is easy to see from the definition that $f_E\in \fA \subset k$. The numbers $f_E$ as $E$ runs through
the irreducible representations of $K(\C)\otz k$ are called formal codegrees of multi-fusion 
category $\C$.

\begin{definition} An algebraic integer $\alpha$ is called a {\em $d-$number} if the ideal it generates
in the ring of algebraic integers is invariant under the action of the absolute Galois group 
$Gal(\bar \BQ/\BQ)$.
\end{definition}

Here is the main result of this note.

\begin{theorem} \label{main}
The formal codegrees of a multi-fusion category are $d-$numbers.
\end{theorem}

We remark that for a general based ring the formal codegrees are not necessarily $d-$numbers,
see Example \ref{rank3}.

We refer the reader to \cite[\S 8.2]{ENO} for the definition of the Frobenius-Perron dimension
$\FPdim(\C)\in \BR$ of a fusion category $\C$. It is known (see {\em loc. cit.}) that $\FPdim(\C)$ 
is an algebraic integer. Let $\phi: \fA \subset \CC$ be an arbitrary ring embedding.
The definition implies immediately that $\phi^{-1}(\FPdim(\C))$
is one of the formal codegrees of $\C$. Thus we have: 

\begin{corollary}\label{FP} For a fusion category $\C$ its Frobenius-Perron dimension 
{\em $\FPdim(\C)$} is a $d-$number. \qed
\end{corollary}

We refer the reader to \cite[Definition 2.2]{ENO} for the definition of the global dimension
$\dim(\C)$ of a fusion category $\C$. Let $\tilde \C$ be the sphericalization of $\C$, see
\cite[Proposition 5.14]{ENO}. It follows immediately from definitions that $2\dim(\C)$ is
one of the formal codegrees of $\tilde \C$. Thus applying Theorem \ref{main} to
the fusion category $\tilde \C$ we have: 

\begin{corollary} For a fusion category $\C$ its global dimension $\dim(\C)$ is a $d-$number. \qed
\end{corollary}

\begin{remark} Let $\C$ be a fusion category. 
Choose a ring embedding $\phi: \fA \subset \CC$ and let $I_{FP}(\C)\subset \fA$
be a principal ideal generated by $\phi^{-1}(\FPdim(\C))$ (note that by Corollary \ref{FP}
the ideal $I_{FP}(\C)$ does not depend on a choice of $\phi$). Similarly, let $I(\C)\subset \fA$
be a principal ideal generated by $\dim(\C)$. Then \cite[Proposition 8.22]{ENO} says that
$I(\C)\subset I_{FP}(\C)$. We don't know any example when $I_{FP}(\C)\ne I(\C)$. 
\end{remark}

For a based ring $K$ a {\em categorification} is a (multi-)fusion 
category $\C$ and an isomorphism of based rings $K\simeq K(\C)$.
Theorem \ref{main} implies that many based rings have no categorifications. 

\begin{example}\label{rank3} Let $k,l,m,n$ be nonnegative
integers satisfying $k^2+l^2=lm+kn+1$ and let $K(k,l,m,n)$ be the based ring with the basis $1,X,Y$
and the multiplication given by (see \cite[\S 3.1]{O2})
$$X^2=1+mX+kY,\; Y^2=1+lX+nY,\; XY=YX=kX+lY.$$
It is easy to compute that any formal codegree of $K(2,1,0,2)$ is a root of (irreducible) polynomial 
$t^3-26t^2+148t-148$; using Lemma \ref{equiv} (v) and Theorem \ref{main} we see that $K(2,1,0,2)$
has no categorifications.
\end{example}

Another condition for a based ring to admit a categorification is
\cite[Theorem 8.51]{ENO}. It states that for a fusion category $\C$, an object $X\in \C$
and an irreducible representation $E$ of $K(\C)\otz k$ the number $\Tr([X],E)$ is a
cyclotomic integer. In particular this implies the same conclusion as in Example \ref{rank3}:
the based ring $K(2,1,0,2)$ has no categorifications. 
On the other hand, Theorem \ref{main} implies that rings
$K(3,2,3,2)$ and $K(8,1,8,7)$ have no categorifications while \cite[Theorem 8.51]{ENO}
gives no conclusion for these rings. On the other hand \cite[Theorem 8.51]{ENO} implies
that ring $K(911,463,1799,232)$ has no categorification while our Theorem \ref{main}
is inconclusive for this ring (this is a unique example with $l\le k<1000$;
I am very grateful to Josiah Thornton for finding this example). 
Finally, both \cite[Theorem 8.51]{ENO}
and Theorem \ref{main} are inconclusive for the based rings of rank 2 while it is known
that many of them have no categorifications, see \cite{O1}.

Recall (see e.g. \cite{Wa}) that a {\em cyclotomic unit} is a product of a root of 1 and units of the form 
$\frac{\zeta^n-1}{\zeta -1}$ where $\zeta$ is a root of 1. Notice that \cite[Theorem 8.51]{ENO}
implies that a formal codegree $f_E$ of a multi-fusion category $\C$ is a cyclotomic 
integer\footnote{Actually in \cite[Theorem 8.51]{ENO} it is assumed that the category $\C$ is fusion
category. But the proof in {\em loc. cit.} extends easily to the case of multi-fusion categories.}.
Thus Theorem \ref{main} and Corollary \ref{cyclo} imply the following result:

\begin{corollary}  Let $f_E$ be a formal codegree of a multi-fusion category $\C$. There exists
a positive integer $m$ such that $f_E^m$ is a rational integer times a cyclotomic unit.
\end{corollary}

Let $\C$ be a fusion category and let $X\in \C$ be a simple object. We refer the reader 
to \cite[\S 2.1]{ENO} for the definition of {\em squared norm} $|X|^2\in k^\times$ (which
was initially defined by M.~M\"uger). Recall that
if $\C$ has a pivotal structure, then $|X|^2=\dim(X)\dim(X^*)=|\dim(X)|^2$, see 
\cite[Proposition 2.9]{ENO}.

\begin{theorem} \label{dim}
 Let $\C$ be a braided fusion category. Then for any simple object
$X\in \C$ the squared norm $|X|^2$ is a $d-$number.
\end{theorem}

\begin{remark} (i) We don't know whether $\FPdim(X)$ is always a $d-$number for a simple 
object $X$ of a braided fusion category.

(ii) There exists a fusion category (related with {\em Haagerup subfactor}, 
see \cite{AH}) with a simple  object of dimension $\frac12(1+\sqrt{13})$ which is not a $d-$number.
Thus the assumption of theorem \ref{dim} that the category $\C$ is braided can not be omitted.
\end{remark}
 
{\bf Acknowledgment.} 
I am deeply grateful to Vladimir Drinfeld for useful discussions and encouragement.
Without his influence,  this note would not have been written. 
Part of the work on this paper was done when the author enjoyed hospitality
of the Institute for Advanced Study; I am happy to thank this institution for 
the excellent research conditions.  
This work was  partially supported by the NSF grant DMS-0602263.

\section{Preliminaries}

\subsection{Multi-fusion categories} The multi-fusion and fusion categories were defined 
in the Introduction. For a multi-fusion category $\C$ we denote by $\O(\C)$ the set
of isomorphism classes of simple objects of $\C$. For an object $X\in \C$ we denote 
by $[X]$ its class in the Grothendieck ring $K(\C)$; for $X\in \C$ and $Y\in \O(\C)$ we denote
by $[X:Y]$ the multiplicity of $Y$ in $X$ (so $[X]=\sum_{Y\in \O(\C)}[X:Y][Y]$).

We say that a multi-fusion category is {\em indecomposable}
(see \cite[\S 2.4]{ENO}) if it is not a direct sum of two nonzero multi-fusion categories
(clearly, any fusion category is automatically indecomposable).

For a multi-fusion category $\C$ let $\Z(\C)$ denote its Drinfeld center, see e.g. \cite{Mu}.
It is known (see \cite[Theorem 2.15]{ENO}) that  $\Z(\C)$ is again a multi-fusion category;
moreover for an indecomposable multi-fusion category $\C$ the category $\Z(\C)$ is a 
fusion category. It follows that the forgetful functor $F: \Z(\C)\to \C$ admits a right
adjoint functor $I: \C \to \Z(\C)$. We will use the following fact, see \cite[Proposition 5.4]{ENO}, 
\cite[Proposition 3.32]{EO}: for any $X\in \C$

\begin{equation}\label{indres}
[F(I(X))]=\sum_{Y\in \O(\C)}[Y][X][Y^*]
\end{equation}

\subsection{Based rings}\label{basedS}
 In this section we recall the basic notions of the theory of based rings
following \cite{Lu1,Lu2}.

\begin{definition}\label{based}
 (\cite[\S 1.1]{Lu1}) A based ring is a pair $(R,B)$ where $R$ is a ring, $B$ is a basis
of $R$ over $\BZ$ such that

(i) All structure constants of $R$ with respect to basis $B$ are nonnegative integers;

(ii) There exists a subset $B_0\subset B$ such that $1=\sum_{b\in B_0}b$;

(iii) Let $\tau: R\to \BZ$ be the group homomorphism such that its restriction to $B\subset R$
is the characteristic function of the subset $B_0\subset B$. There exists an anti-involution
$r\mapsto \tilde r$ of the ring $R$ which preserves the subset $B$ and such that $\tau(bb')=0$ for
$b'\in B\setminus \{ \tilde b\}$ and $\tau(b\tilde b)=1$ for any $b\in B$.
\end{definition} 

\begin{example}\label{groth}
(i) Let $\C$ be a multi-fusion category. Then its Grothendieck ring $K(\C)$ with
a basis given by the classes of simple objects is an example of based ring.

(ii) Here is a special case of (i).
Let $G$ be a finite group. Then its group ring $\BZ[G]$ together with a basis $G\subset \BZ[G]$
is an example of based ring. 
\end{example}

For a based ring $(R,B)$ its {\em rank} is by definition the cardinality of the set $B$. In this 
note we will consider only based rings $(R,B)$ of finite rank.
This implies that the $k-$algebra $R\otz k$ is semisimple, see \cite[1.2(a)]{Lu1}.
In what follows a {\em representation} of $R$ is finite dimensional representation
over the field $k$. Let $\Irr(R)$ denote the set of isomorphism classes of irreducible 
representations of $R$.

Let $(R,B)$ be a based ring and let $\tau: R\to \BZ$ be as in Definition \ref{based} (iii).
We extend $\tau$ by linearity to the map $\tau_k: R\otz k\to k$.
The pair $(R\otz k,\tau_k)$ is an example of {\em algebra with a trace form}, see \cite[Chapter 19]{Lu2}.
This means that the bilinear form $(x,y)=\tau_k(xy)$ on $R\otz k$ is non-degenerate; thus we
can identify $R\otz k$ with the dual space $(R\otz k)^*$. For $E\in \Irr(R)$ let $\alpha_E\in R\otz k$ 
be the element corresponding to a linear function $\Tr(?,E)\in (R\otz k)^*$. 

\begin{lemma}\label{lu} {\em (see \cite[Proposition 19.2]{Lu2})}
We have explicitly $\alpha_E=\sum_{b\in B}\Tr(b,E)\tilde b$.
The element $\alpha_E\in R\otz k$ is central; it acts by a scalar 
$f_E\in k$ on $E$ and by zero on any other {\em $E'\in \Irr(R)$}.\qed
\end{lemma}

The scalar $f_E$ is called a {\em formal codegree} of representation $E$.
It is clear from definition that $f_E$ is an algebraic integer.

\begin{example}\label{one}
 Let $\chi: R\to k$ be a ring homomorphism. Let $E_\chi$ be
the corresponding one dimensional representation of $R$. Then clearly
$$f_{E_\chi}=\sum_{b\in B}\chi(b)\chi(\tilde b)=\sum_{b\in B}|\chi(b)|^2.$$
\end{example}

In the case of based ring $K(\C)$ from Example \ref{groth} (i) we say that the scalars $f_E$ 
are {\em formal codegrees} of the multi-fusion category $\C$.

\begin{example}\label{ZG} It is easy to compute that for the based ring from Example \ref{groth} (ii)
we have $f_E=\frac{|G|}{\dim(E)}$. This is a motivation for our terminology.
\end{example}

\begin{lemma}\label{alpha}
 For any based ring $(R,B)$ an element $\sum_{b\in B}b\tilde b$ is central.
It acts on {\em $E\in \Irr(R)$} as $f_E\dim(E)\id_E$.
\end{lemma}

\begin{proof} Consider the element 
$\alpha=\sum_{E\in \Irr(R)}\dim(E)\alpha_E\in R\otz k$.
By Lemma \ref{lu} $\alpha$ is central and it acts on $E\in \Irr(R)$ as 
$f_E\dim(E)\id_E$.

Since the regular representation $R\otz k$ of $R$ decomposes as 
$\oplus_{E\in \Irr(R)}E^{\dim(E)}$ we see
that $$\alpha =\sum_{E\in \Irr(R)}\dim(E)\sum_{b\in B}\Tr(b,E)\tilde b=
\sum_{b\in B}\Tr(b,R\otz k)\tilde b=\sum_{b,b'\in B}\tau(bb'\tilde b')\tilde b=
\sum_{b'\in B}b'\tilde b'$$ The Lemma is proved.
\end{proof}

\subsection{$d-$numbers} Let $\fA^\times \subset \fA$ be the subset of units.

\begin{lemma}\label{equiv}
 For an algebraic integer $\alpha$ the following conditions are equivalent:

(i) $\alpha$ is a $d-$number;

(ii) For any $g\in Gal(\bar \BQ/\BQ)$ we have $\frac{\alpha}{g(\alpha)}\in \fA$;

(iii) For any $g\in Gal(\bar \BQ/\BQ)$ we have $\frac{\alpha}{g(\alpha)}\in \fA^\times$;

(iv) There exists a positive integer $m$ such that $\alpha^m\in \BZ \cdot \fA^\times$;

(v) Let $p(x)=x^n+a_1x^{n-1}+\ldots +a_n$ be the minimal polynomial of $\alpha$ over
$\BQ$ (so, $a_i\in \BZ$). Then for any $i=1, \ldots,n$ the number $(a_i)^n$ is divisible by 
$(a_n)^i$;

(vi) There exists a polynomial $p(x)=x^m+a_1x^{m-1}+\ldots +a_m\in \BZ[x]$ such that $p(\alpha)=0$
and $(a_i)^m$ is divisible by $(a_m)^i$.
\end{lemma}

\begin{proof} The implications $(iii)\imply (i)\imply (ii)$, $(iv)\imply (iii)$ and $(v)\imply (vi)$ are obvious.
Next $(ii)\imply (iii)$ since $\left(\frac{\alpha}{g(\alpha)}\right)^{-1}=
g\left(\frac{\alpha}{g^{-1}(\alpha)}\right)$. Let $\alpha_1=\alpha, \alpha_2, \ldots, \alpha_n$ be the
conjugates of $\alpha$. Then $(iii)$ implies that $\alpha^n\in (\prod_{i=1}^n\alpha_i)\fA^\times$;
in particular $(iii)\imply (iv)$. Similarly $(iii)$ implies that 
$\frac{\alpha_j}{\sqrt[n]{\prod_{i=1}^n\alpha_i}}\in \fA^\times$ for any $j=1,\ldots ,n$; hence by
Vieta's theorem $(iii)\imply (v)$. Finally $(vi)$ implies that $\frac{\alpha}{\sqrt[m]{a_m}}$ is a root
of polynomial $\tilde p(x)=x^m+\frac{a_1}{\sqrt[m]{a_m}}x^{m-1}+\ldots +\frac{a_{m-1}}{(\sqrt[m]{a_m})^{m-1}}x+1\in \fA[x]$; hence $\frac{\alpha}{\sqrt[m]{a_m}}\in \fA^\times$ and $(vi)\imply (iv)$.
\end{proof}

\begin{corollary} \label{twice}
Let $\alpha, \beta \in \fA$ and $\frac{\alpha}{\beta}\in \BQ$. Then $\alpha$ is
a $d-$number if and only if $\beta$ is. \qed
\end{corollary}

\begin{corollary}\label{square}
 Let $\alpha \in \fA$. Then $\alpha$ is a $d-$number if and only if $\alpha^2$ is. \qed
\end{corollary}

\begin{corollary}\label{cyclo}
 Let $\alpha$ be a cyclotomic algebraic integer. Then $\alpha$ is a $d-$number
if and only if there exists a positive integer $m$ such that $\alpha^m$ is an integer times 
a cyclotomic unit.
\end{corollary}

\begin{proof} This follows immediately from Lemma \ref{equiv} (iv) since it is known that
the subgroup of cyclotomic units is of finite index in the group of all units of a cyclotomic field,
see \cite[Chapter 8]{Wa}.
\end{proof}

\section{Proofs}

\subsection{Key Lemma}
Let $\C$ be a multi-fusion category and let $\Z(\C)$ be its Drinfeld center.
The forgetful functor $F: \Z(\C)\to \C$ induces a ring homomorphism $f: K(\Z(\C))\to K(\C)$. 
It is clear that the image of this map is contained in the center $Z(K(\C))$ of the ring $K(\C)$. 

Let $E\in \Irr(K(\C))$. By Schur's Lemma for any
$x\in K(\Z(\C))$ the element $f(x)\in K(\C)$ acts on $E$ by a scalar $\chi_E(x)\in k$. Extending
$\chi_E$ by linearity we get a homomorphism $\chi_E: K(\Z(\C))\otz k\to k$. Let $\hat E$ be
the corresponding one dimensional (hence irreducible) representation of $K(\Z(\C))$. 

\begin{lemma} \label{key}
We have $f_{\hat E}=f_E^2$.
\end{lemma}

\begin{proof} By Example \ref{one} an element
 $\sum_{Y\in \O(\Z(\C))}[F(Y)]\cdot [F(Y^*)]\in K(\C)$ acts on the representation $E$
 as $f_{\hat E}\id_E$. We compute

\begin{eqnarray*}
\sum_{Y\in \O(\Z(\C))}[F(Y)][F(Y^*)]=\sum_{Y\in \O(\Z(\C))}\sum_{X\in \O(\C)}[F(Y):X][X][F(Y^*)]=\\
\sum_{Y\in \O(\Z(\C))}\sum_{X\in \O(\C)}[I(X):Y][X][F(Y^*)]=\\
\sum_{Y\in \O(\Z(\C)),X\in \O(\C)}[I(X)^*:Y^*][X][F(Y^*)]=\\
\sum_{X\in \O(\C)}[X][F(I(X)^*)]=\sum_{X\in \O(\C)}[F(I(X))][X^*].
\end{eqnarray*}

Since $[F(I(X))]\in Z(K(\C))$ it acts by a scalar on the irreducible representation $E$. To compute
this scalar we calculate the trace of $[F(I(X))]$ using \eqref{indres} and Lemma \ref{alpha}:

\begin{eqnarray*}
\Tr \left([F(I(X))],E\right)=\Tr \left(\sum_{Z\in \O(\C)}[Z^*][X][Z],E\right)=\\
=\Tr \left(\sum_{Z\in \O(\C)}[Z][Z^*][X],E\right)=
f_E\dim(E)\Tr([X],E)
\end{eqnarray*}

Hence $\sum_{Y\in \O(\Z(\C))}[F(Y)][F(Y^*)]$ acts on $E$ as 
$\sum_{X\in \O(\C)}f_E\Tr([X],E)[X^*]=f_E^2\id_E$. The Lemma is proved.
\end{proof}

\subsection{Reduction of the proof of Theorem \ref{main} to the case when $\C$ is modular} We
can assume that the category $\C$ is indecomposable
since the formal codegrees of a direct
sum of multi-fusion categories are clearly just formal codegrees of the summands.

The Drinfeld center of an indecomposable multi-fusion category is a fusion category, so
Lemma \ref{key} and Corollary \ref{square} imply that we can assume that the category $\C$
is a fusion category. 

For a fusion category $\C$ let $\tilde C$ be its {\em sphericalization}, see \cite[Proposition 5.14]{ENO}.
This is a spherical fusion category together with forgetful functor $\tilde \C \to \C$ which induces 
a surjective ring homomorphism $K(\tilde C)\to K(\C)$. In particular any representation $E\in \Irr(K(\C))$
can be considered as a representation $\tilde E\in \Irr(K(\tilde \C))$. It is easy
to see that $f_{\tilde E}=2f_E$ for any $E\in \Irr(K(\C))$. So by Corollary \ref{twice} we can
assume that the category $\C$ is spherical fusion category.

Finally the Drinfeld center of a spherical fusion category is a modular tensor category, see \cite{Mu}. 
So applying Lemma \ref{key} and Corollary \ref{square} once again we see that it is
enough to prove Theorem \ref{main} for a modular tensor category $\C$.

\subsection{Proof of Theorem \ref{main} in the case when $\C$ is modular} \label{mod}
Let $\C$ be a modular tensor category.
Let $\O(\C)=\{ X_i\}_{i\in I}$, let $0\in I$ be such that $X_0=\be$ and let $i\mapsto i^*$ be the
involution such that $X_i^*=X_{i^*}$.
Let $S=(s_{ij})_{i,j\in I}$ be the $S-$matrix of the modular tensor category $\C$,
see e.g. \cite{BK}. For any $i\in I$ the assignment $\chi_i([X_j])=\frac{s_{ij}}{s_{0i}}$ extends 
by linearity to a ring homomorphism $K(\C)\to k$, moreover any such homomorphism is
$\chi_i$ for a well defined $i\in I$, see {\em loc. cit.} In particular we have
\begin{equation}\label{integer}
\frac{s_{ij}}{s_{0i}}=\frac{s_{ji}}{s_{0i}}\in \fA \; \mbox{for all}\; i,j\in I
\end{equation}
Since $K(\C)$ is commutative, any
irreducible representation of $K(\C)$ is of the form $E_{\chi_i}$, see Example \ref{one}. 
Using Example \ref{one} and the identity $\sum_{j\in I}s_{ij}s_{j^*k}=\delta_{ik}$ 
(see \cite[3.1.17]{BK}), one computes easily $f_{E_{\chi_i}}=\frac1{s_{0i}^2}$. 
Thus for $g\in Gal(\bar \BQ/\BQ)$ we have $\frac{f_{E_{\chi_i}}}{g(f_{E_{\chi_i}})}=\frac{g(s_{0i})^2}{s_{0i}^2}$. 
On the other hand it is well known (see \cite{dBG,CG} or \cite[Appendix]{ENO})
that there exists a permutation $g:I\to I$ such that $g(s_{ij})=\pm s_{g(i)j}$. In particular,
$\frac{f_{E_{\chi_i}}}{g(f_{E_{\chi_i}})}=\frac{s_{g(0)i}^2}{s_{0i}^2}\in \fA$ and we are done 
by \eqref{integer} and Lemma \ref{equiv} (ii).\qed

\subsection{Proof of Theorem \ref{dim}} First we reduce proof of Theorem \ref{dim} to the case
when $\C$ is modular. Let $\tilde C$ be the sphericalization of the category $\C$; then $\tilde \C$
is a braided fusion category with a spherical structure and for any $X\in \O(\C)$ there exists
$\tilde X\in \O(\tilde \C)$ such that $|X|^2=\dim(\tilde X)^2$. Thus by Corollary \ref{square}
it is enough to prove the following:

(a) Let $\C$ be a spherical braided fusion category. Then $\dim(X)$ is a $d-$number for any 
$X\in \O(\C)$.

Next for any braided fusion category $\C$ we have a fully faithful functor $\C \to \Z(\C)$;
this functor preserves dimensions in a case when $\C$ is spherical. Since the Drinfeld center
of a spherical fusion category is a modular tensor category (\cite{Mu}) we see that (a) is a consequence
of the following:

(b) Let $\C$ be a modular tensor category. Then $\dim(X)$ is a $d-$number for any 
$X\in \O(\C)$.

Let us prove (b). In the notation of \S \ref{mod} we have $\dim(X_i)=\frac{s_{0i}}{s_{00}}$,
see \cite[\S 3.1]{BK}. Thus for $g\in Gal(\bar \BQ/\BQ)$ we have $\frac{\dim(X_i)}{g(\dim(X_i))}=
\frac{s_{0i}g(s_{00})}{g(s_{0i})s_{00}}=
\pm \frac{s_{g(i)g^{-1}(0)}}{s_{0g(i)}}\cdot \frac{s_{0g(0)}}{s_{00}}$. Thus by \eqref{integer} 
we have $\frac{\dim(X_i)}{g(\dim(X_i))}\in \fA$ and we are done by Lemma \ref{equiv} (ii).\qed

\bibliographystyle{ams-alpha}

\begin{thebibliography}{A} 

\bibitem[AH]{AH} M.~Asaeda, U.~Haagerup, \textit{Exotic subfactors of finite depth with Jones 
indices $(5+\sqrt{13})/2$ and $(5+\sqrt{17})/2$},  Comm. Math. Phys.  \textbf{202}  (1999),  no. 1, 1--63. 

\bibitem[BK]{BK} B.~Bakalov, A.~Kirillov Jr.,
\textit{Lectures on Tensor categories and modular functors}, 
AMS, (2001).

\bibitem[CG]{CG} A.~Coste, T.~Gannon,
\textit{Remarks on Galois symmetry in rational conformal field theories},
Phys. Lett. B \textbf{323} (1994), 316-321.

\bibitem[dBG]{dBG} J.~de Boer, J.~Goeree, \textit{Markov traces and $II_1$ factors in
conformal field theory}, Comm. Math. Phys. \textbf{139} (1991), 267-304.


\bibitem[ENO]{ENO} P.~Etingof, D.~Nikshych, and V.~Ostrik,
\textit{On fusion categories}, Annals of Mathematics 
\textbf{162} (2005), 581-642.

\bibitem[EO]{EO} P.~Etingof, V.~Ostrik,
\textit{Finite tensor categories},
Moscow Math.\ Journal \textbf{4} (2004), 627-654.


\bibitem[L1]{Lu1} G.~Lusztig,
\textit{Leading coefficients of character values of Hecke algebras},  The Arcata Conference on Representations of Finite Groups (Arcata, Calif., 1986),  235--262,
Proc. Sympos. Pure Math., {\bf 47}, Part 2, Amer. Math. Soc., Providence, RI, 1987. 

\bibitem[L2]{Lu2} G.~Lusztig,
\textit{Hecke algebras with unequal parameters}, CRM Monograph Series, \textbf{18}. American 
Mathematical Society, Providence, RI, 2003. vi+136 pp.

\bibitem [M]{Mu} M.~M\"uger,
\textit{From subfactors to categories and topology. II. 
The quantum double of tensor categories and subfactors},
J.\ Pure Appl.\ Algebra  \textbf{180}  (2003),  no. 1-2, 159--219.

\bibitem[O1]{O1}  V.~Ostrik, \textit{Fusion categories of rank 2}, Math. Res. Lett. 
\textbf{10} (2003), no. 2-3, 177--183.

\bibitem[O2]{O2} V.~Ostrik, \textit{Pre-modular categories of rank 3}, Mosc. Math. J. \textbf{8} 
(2008), no. 1, 111--118.

\bibitem[W]{Wa} L.~C.~Washington, \textit{Introduction to cyclotomic fields}, Graduate Texts in Mathematics, \textbf{83}. Springer-Verlag, New York, 1997. xiv+487 pp.

\end{thebibliography}

\end{document}